\documentclass[11pt]{article}

\usepackage{amsmath,amsfonts,amssymb,amsthm,bbm}
\usepackage{graphics,epsfig}
\usepackage{color}
\usepackage{array}

\newtheorem{theorem}{Theorem}[section]

\newtheorem{lemma}{Lemma}[section]

\numberwithin{equation}{section}

\allowdisplaybreaks

\renewcommand\today{\number\year/\number\month/\number\day}

\def\qed{\hfill$\Box$\medskip}

\begin{document}

\noindent\makebox[60mm][l]{\tt {\large Version:~\today}}

\bigskip
\noindent{
\begin{center}
\LARGE\bf Hitting Time Distribution for  Skip-Free Markov Chains: A Simple  Proof
\end{center}
}

\noindent{
\begin{center}Wenming Hong\footnote{School of Mathematical Sciences
\& Laboratory of Mathematics and Complex Systems, Beijing Normal
University, Beijing 100875, P.R. China. Email: wmhong@bnu.edu.cn}  \ \ \
 Ke Zhou\footnote{ School of Mathematical Sciences
\& Laboratory of Mathematics and Complex Systems, Beijing Normal
University, Beijing 100875, P.R. China. Email:zhouke@mail.bnu.edu.cn}
\end{center}
}

\vspace{0.1 true cm}

\begin{center}
\begin{minipage}[c]{12cm}
\begin{center}\textbf{Abstract}\end{center}
\bigskip
 A well-known theorem  for an irreducible skip-free chain with absorbing state $d$, under some conditions, is that the hitting (absorbing) time of state $d$ starting from state $0$  is distributed as the sum of $d$ independent geometric (or exponential) random variables. The purpose of this paper is to present a direct and simple proof of the theorem in the cases of both discrete  and continuous time  skip-free Markov chains. Our proof is to calculate directly the generation functions (or Laplace transforms) of hitting times in terms of the iteration method.\\

\mbox{}\textbf{Keywords:}\quad skip-free, random walk, birth and death chain, absorbing time, hitting time, eigenvalues, recurrence equation.\\
\mbox{}\textbf{Mathematics Subject Classification (2010)}:  60E10, 60J10, 60J27, 60J35.
\end{minipage}
\end{center}


\section{ Introduction \label{s1}}
The skip-free Markov chain on $\mathbb{Z^{+}}$ is a process for which upward jumps may be only of unit size, and there is no restriction on downward jumps. If a chain start at $0$, and we suppose $d$ is an absorbing state.  An interesting property for the chain is that the hitting time of state $d$ is distributed as a sum of $d$ independent geometric (or exponential) random variables.

There are many authors give out different proofs to the results. For the birth and death chain, the well-known results can be traced back to  Karlin and McGregor \cite{KM}, Keilson \cite{Keillog}, \cite{Keil}. Kent and Longford \cite{Kent} proved the result for the discrete time version (nearest random walk) although  they have not specified the result as usual form (section 2, \cite{Kent}). Fill \cite{F1} gave the first stochastic proof to both nearest random walk and birth and death chain cases via duality which was established in \cite{DF}. Diaconis and Miclo \cite{DM} presented another probabilistic proof for birth and death chain. Very recently, Gong, Mao and Zhang \cite{Mao} gave a similar result in the case that the state space is $\mathbb{Z^{+}}$, they use the well established result to  determine all the eigenvalues or the spectrum of the generator.

For the  skip-free chain, Brown and Shao \cite{BS} first proved  the result in continuous time situation. By using the duality,  Fill \cite{F2} gave a stochastic proof to both discrete  and continuous time cases. The purpose of this paper is to present a direct and simple proof of the theorem in the cases of both discrete  and continuous time  skip-free Markov chains. Our proof is to calculate directly the generation functions (or Laplace transforms) of hitting times in terms of the iteration method.

\begin{theorem}For the discrete-time skip-free random walk:\\
Consider an irreducible skip-free random walk with transition probability $P$ on $\{0,1,\cdots, d\}$ started at $0$, suppose $d$ is an absorbing state. Then the hitting time of state $d$ has the generation function
\begin{equation*}
\varphi_{d}(s)=\prod_{i=0}^{d-1} \left[\frac{(1-\lambda_i)s}{1-\lambda_i s} \right],
\end{equation*}
where $\lambda_0, \cdots, \lambda_{d-1}$ are the $d$ non-unit eigenvalues of $P$.

In particular, if all of the eigenvalues are real and nonnegative,  then the hitting time is distributed as the sum of $d$ independent geometric random variables  with parameters $1-\lambda_i$.\\
\end{theorem}

\begin{theorem} For the  skip-free birth and death chain:\\
Consider an irreducible  skip-free birth and death chain with generator $Q$ on $\{0,1,\cdots, d\}$ started at $0$, suppose $d$ is an absorbing state. Then the hitting time of state $d$ has the Laplace transform
\begin{equation*}
\varphi_{d}(s)=\prod_{i=0}^{d-1}\frac{\lambda_i}{1-\lambda_i s},
\end{equation*}
where $\lambda_i$ are the $d$ non-zero eigenvalues of $-Q$.

 In particular, if all of the eigenvalues are real and nonnegative, then the hitting time is distributed as the sum of $d$ independent exponential random variables with parameters $\lambda_i$.
 \end{theorem}

\section{ Proof of Theorem 1.1\label{s2}}
Define the transition probability matrix $P$ as
\begin{equation*}
P=\left(
 \begin{array}{ccccccc}
  r_0 & p_0 & \\
   q_{1,0} & r_1 & p_1 & \\
   \ddots & \ddots & \ddots & & \ddots & \ddots & \\
   q_{d-1,0}& q_{d-1,1}& q_{d-1,2}& & \cdots &  r_{d-1} & p_{d-1} \\
    &  &  &  &  &  & 1\\
     \end{array}
      \right)_{(d+1)\times(d+1)},
\end{equation*}and for $0\leq n\leq d-1$, $P_{n}$ denote the first $n+1$ rows and first $n+1$ lines of $P$.

Let $\tau_{i,i+j}$ be the hitting time of state $i+j$ starting from $i$. By the Markov property,
we have 
\begin{equation}
\label{2.1}\tau_{i, i+j}=\tau_{i,i+1}+\tau_{i+1,i+2}+\cdots+\tau_{i+j-1, i+j}.
\end{equation}
If  $f_{i,i+1}(s)$ is the generation function of $\tau_{i,i+1}$,
\begin{equation*}f_{i,i+1}(s)=\mathbb{E}s^{\tau_{i,i+1}}~~~\text{for}~ 0\leq i\leq d-1, \end{equation*}
(\ref{2.1}) says that \begin{equation*}
f_{i,i+j}(s)=f_{i,i+1}(s)\cdot f_{i+1,i+2}(s)\cdot\cdots \cdot f_{i+j-1,i+j}(s),~~\text{for}~ 1\leq j\leq d-i.
\end{equation*}

Let 
\begin{equation*}g_{0,0}(s)=1,~~~g_{i,i+j}(s)=\frac{p_{i}p_{i+1}\cdots p_{i+j-1}}{f_{i,i+j}(s)}s^{j},~~\text{for}~ 1\leq j\leq d-i.\end{equation*}

\begin{lemma}Define $I_n$ as a $(n+1)\times(n+1)$ identity matrix. We have
\begin{equation}\label{lm}
g_{0,n+1}(s)=\text{det}(I_n-sP_{n}),~~~\text{for}~ 0\leq n\leq d-1.
\end{equation}
\end{lemma}
\begin{proof}
We will give a key recurrence to proof this lemma. By decomposing the first step ,the generation function of $\tau_{n,n+1}$ satisfy
 \begin{equation} \label{key}
 \begin{split} f_{n,n+1}(s)&=r_{n}sf_{n,n+1}(s)+p_{n}s+q_{n,n-1}sf_{n-1,n+1}(s)+q_{n,n-2}sf_{n-2,n+1}(s)+\\
&~~~~~\cdots+q_{n,0}sf_{0,n+1}(s).
\end{split}
\end{equation}
Recall the definition of $g_{i,i+j}(s)$, substitute it into the formula above, we have
 \begin{equation}\label{f 2.5}\begin{split}g_{0,n+1}(s)&=(1-r_{n}s)g_{0,n}(s)-q_{n,n-1}s^{2}p_{n-1}g_{0,n-1}(s)-q_{n,n-2}s^{3}p_{n-2}p_{n-1}g_{0,n-2}(s)-\\
&~~~~~\cdots-q_{n,0}s^{n+1}p_{0}p_{1}\cdots p_{n-1}g_{0,0}(s). \end{split}\end{equation}  Use the notation $A_{i,j}$ to denote the algebraic complement of the position $(i+1,j+1)$ in $I_n-sP_{n}$. By expanding the bottom row of the matrix, we obtain\begin{equation}\label{f 2.6}\text{det}(I_n-sP_{n})=(1-r_{n}s)A_{n,n}-q_{n,n-1}sA_{n,n-1}-q_{n,n-2}sA_{n,n-2}-\cdots-q_{n,0}sA_{n,0}.\end{equation}
By some calculation, we can deduce \begin{equation*}A_{n,n}=\text{det}(I_{n-1}-sP_{n-1}),~~A_{n,0}=p_{0}p_{1}\cdots p_{n-1}s^{n},\end{equation*} and for $1\leq i< n$,\begin{equation*} A_{n,n-i}=p_{n-i}p_{n-i+1}\cdots p_{n-1}s^{i}\text{det}(I_{n-i-1}-sP_{n-i-1}).\end{equation*}  Now we prove the lemma by  induction. At first, $g_{0,1}(s)=\text{det}(I_0-sP_{0})=1-r_{0}s$. If (\ref{lm}) holds for $n< k$, we calculate $g_{0,k+1}(s)$. By (\ref{f 2.5}),
 \begin{equation*}\begin{split}g_{0,k+1}(s)&=(1-r_{k}s)\text{det}(I_{k-1}-sP_{k-1})-q_{k,k-1}s^{2}p_{k-1}\text{det}(I_{k-2}-sP_{k-2})-\\
&~~~~~q_{k,k-2}s^{3}p_{k-2}p_{k-1}\text{det}(I_{k-3}-sP_{k-3})-\cdots-q_{k,0}s^{k+1}p_{0}p_{1}\cdots p_{k-1}. \end{split}\end{equation*}
Formula (\ref{f 2.6}) tell us that $g_{0,k+1}(s)=\text{det}(I_k-sP_{k})$. The proof is complete.
\end{proof}

\noindent{\it  Proof of Theorem 1.1 }~~Denote $\varphi_{d}(s)$ the generation function of $\tau_{0,d}$, then $\varphi_{d}(s)=f_{0,d}(s)$. By (\ref{2.1}) and (\ref{lm})
 ,  we have
\begin{equation*}\begin{split}\varphi_{d}(s)&=f_{0,1}(s)f_{1,2}(s)\cdots f_{d-1,d}(s)\\
&=\frac{p_{0}p_{1}\cdots p_{d-1}s^{d}}{g_{0,d}(s)}=\frac{p_{0}p_{1}\cdots p_{d-1}s^{d}}{\text{det}(I_{d-1}-sP_{d-1})}.\end{split}\end{equation*}
 It is easy to prove that $1$ is the unique unit eigenvalue of $P$, and for $i=0,2\cdots d-1$, $\lambda_i$ are the $d$ non-unit eigenvalues. So \begin{equation*}\text{det}(I_{d-1}-sP_{d-1})=(1-\lambda_{0}s)(1-\lambda_{1}s)\cdots(1-\lambda_{d-1}s).\end{equation*}
By Lemma 2.1 and the definition of $g_{0,d}(s)$,\begin{equation*}\begin{split}
p_{0}p_{1}\cdots p_{d-1}&=s^{-d}f_{0,d}(s)\text{det}(I_{d-1}-sP_{d-1})\\
&=s^{-d}f_{0,d}(s)(1-\lambda_{0}s)(1-\lambda_{1}s)\cdots(1-\lambda_{d-1}s).\end{split}\end{equation*}
Let $s=1$, because $f_{0,d}$ is a generation function, $f_{0,d}(1)=1$. Then\begin{equation*}
p_{0}p_{1}\cdots p_{d-1}=(1-\lambda_{0})(1-\lambda_{1})\cdots(1-\lambda_{d-1}).\end{equation*}
 As a consequence we have\begin{equation*}\begin{split}
\varphi_{d}(s)&=\frac{(1-\lambda_{0})(1-\lambda_{1})\cdots(1-\lambda_{d-1})s^{d}}{(1-\lambda_{0}s)(1-\lambda_{1}s)\cdots(1-\lambda_{d-1}s)}\\
&=\prod_{i=0}^{d-1} \left[\frac{(1-\lambda_i)s}{1-\lambda_i s} \right].\end{split}\end{equation*}
\qed
\section{ Proof of Theorem 1.2\label{s3}}
 Denote the generator $Q$ of the skip-free Markov chain  as
\begin{equation*}
Q=\left(
\begin{array}{ccccccc}
        -\gamma_0 & \alpha_0 & \\
   \beta_{1,0} & -\gamma_1 & \alpha_1 & \\
   \ddots & \ddots & \ddots & & \ddots & \ddots & \\
   \beta_{d-1,0}& \beta_{d-1,1}& \beta_{d-1,2}& & \cdots &  -\gamma_{d-1} & \alpha_{d-1} \\
    &  &  &  &  &  & 0\\
     \end{array}
      \right)_{(d+1)\times(d+1)},
\end{equation*}and for $0\leq n\leq d-1$, $Q_{n}$ denote the sub-matrix of the first $n+1$ rows and first $n+1$ lines of $Q$. $\tau_{i,i+j}$ be the hitting time of state $i+j$ starting from $i$. The idea of proof  is similar as Theorem 1.1. We just give a briefly description here.


It is well known that the skip-free chain on the finite state  has an simple structure.
 The process start at $i$, it stay there with an $\mbox{Exponential}~(\gamma_i) $ time,
 then jumps to $i+1$ with probability $\frac{\alpha_i}{\gamma_i}$, to $i-k$
 with probability $\frac{\beta_{i,i-k}}{\gamma_i}~(1\leq k\leq i).$  Let   $\widetilde{f}_{i,i+j}(s)$ be
 the Laplace transform of $\tau_{i,i+j}$,

\begin{equation*}\widetilde{f}_{i,i+j}(s)=\mathbb{E}e^{-s\tau_{i,i+j}}. \end{equation*}
Recall that if a random variable $\xi\sim\mbox{Exponential}~(\theta)$, \begin{equation*}\mathbb{E}e^{-s\xi}=\frac{\theta}{\theta+s}. \end{equation*}
By decomposing the trajectory at the first jump, 
\begin{equation*} \label{keykey}\begin{split} \widetilde{f}_{n,n+1}(s)&=\frac{\gamma_n}{\gamma_n+s}\frac{\alpha_n}{\gamma_n} +\frac{\gamma_n}{\gamma_n+s}\frac{\beta_{n,n-1}}{\gamma_n}\widetilde{f}_{n-1,n+1}(s)+\frac{\gamma_n}{\gamma_n+s}\frac{\beta_{n,n-2}}{\gamma_n}\widetilde{f}_{n-2,n+1}(s)+\\
&~~~~~\cdots+\frac{\gamma_n}{\gamma_n+s}\frac{\beta_{n,0}}{\gamma_n}\widetilde{f}_{0,n+1}(s)\\
&=\frac{\alpha_n}{\gamma_n+s}+\frac{\beta_{n,n-1}}{\gamma_n+s}\widetilde{f}_{n-1,n+1}(s)+\frac{\beta_{n,n-2}}{\gamma_n+s}\widetilde{f}_{n-2,n+1}(s)+\cdots+\frac{\beta_{n,0}}{\gamma_n+s}\widetilde{f}_{0,n+1}(s).\end{split}\end{equation*}
 Define \begin{equation*}\widetilde{g}_{0,0}(s)=1,~~~\widetilde{g}_{i,i+j}(s)=\frac{\alpha_{i}\alpha_{i+1}\cdots \alpha_{i+j-1}}{\widetilde{f}_{i,i+j}(s)},~~\text{for}~ 1\leq j\leq d-i.\end{equation*}

  The following lemma can be  proved by use the method similar  as  Lemma 2.1, we omit the details.
\begin{lemma}
\begin{equation*} \widetilde{g}_{0,n+1}(s)=\text{det}(sI_{n}-Q_{n}),~~~\text{ for}~ 0\leq n\leq d-1.\end{equation*}
\end{lemma}
Then we can calculate the Laplace transform of $\tau_{0,d}$, recall that  $\lambda_0, \dots, \lambda_{d-1}$ are the $d$ non-zero eigenvalues of $-Q$, we can see they are not equal to $0$ easily. And we have \begin{equation*}\alpha_{0}\alpha_{1}\cdots \alpha_{d-1}= \lambda_{0}\lambda_{1}\cdots\lambda_{d-1}.\end{equation*} So \begin{equation*}\begin{split}\varphi_{d}(s)&=\widetilde{f}_{0,1}(s)\widetilde{f}_{1,2}(s)\cdots \widetilde{f}_{d-1,d}(s)\\&=\frac{\alpha_{0}\alpha_{1}\cdots \alpha_{d-1}}{\text{det}(sI_{d-1}-Q_{d-1})}=\prod_{i=0}^{d-1}\frac{\lambda_i}{1-\lambda_i s}
\end{split}\end{equation*}
complete the proof. \qed

\end{document}